\documentclass[journal, twoside]{IEEEtran}
\usepackage{amsmath,amsfonts}
\usepackage{algorithmic}
\usepackage{algorithm}
\usepackage{array}
\usepackage{textcomp}
\usepackage{stfloats}
\usepackage{url}
\usepackage{verbatim}
\usepackage{graphicx}
\usepackage{cite}
\usepackage{amsthm}
\usepackage{multirow}
\usepackage[caption=false,
labelfont=sf,textfont=sf]{subfig}
\def\BibTeX{{\rm B\kern-.05em{\sc i\kern-.025em b}\kern-.08em
    T\kern-.1667em\lower.7ex\hbox{E}\kern-.125emX}}
\newtheorem{theorem}{Theorem}[]

\theoremstyle{definition}

\begin{document}

\title{An Exact Algorithm for Load-Dependent Traveling Salesman Problem for Unmanned Aerial Vehicle Package Delivery
\\
\thanks{Deepak Prakash Kumar is with the Department of Electrical Engineering and Computer Science, University of California, Irvine, CA 92697, USA (e-mail: {\tt\small deepakprakash1997@gmail.com}).

Saurabh Belgaonkar and Swaroop Darbha are with the Department of Mechanical Engineering, Texas A\&M University, College Station, TX 77843, USA (e-mail: {\tt\small saurabhbelgaonkar@tamu.edu, dswaroop@tamu.edu}).

Sivakumar Rathinam is with the Computer Science and Engineering Department and with the Mechanical Engineering Department, Texas A\&M University, College Station, TX 77843, USA (e-mail: {\tt\small srathinam@tamu.edu}).

David Casbeer is with the Control Science Center, Air Force Research Laboratory, Wright-Patterson Air Force Base, OH 45433 USA (e-mail:
{\tt\small david.casbeer@us.af.mil}).

DISTRIBUTION STATEMENT A. Approved for public release. Distribution is unlimited. AFRL-2025-5307; Cleared 11/25/2025.}
}

\author{Deepak Prakash Kumar, Saurabh Belgaonkar, Sivakumar Rathinam, Swaroop Darbha, David W. Casbeer}

\markboth{IEEE Transactions.}%
{Kumar \MakeLowercase{\textit{et al.}}: Exact Algorithm for Load-Dependent TSP for UAV Package Delivery}

\maketitle

\bstctlcite{IEEEexample:BSTcontrol}

\begin{abstract}
    In this article, we present a novel formulation for the load-dependent traveling salesman problem (LD-TSP), in which travel cost (or energy expended) depends on the vehicle’s current load. This problem is relevant for package delivery and urban air mobility, where vehicles must transport and drop cargo at specified locations. The challenge lies in modeling the cost, which varies with both route sequence and onboard load. Our key contributions are: (i) formulating an energy dissipation model and proving energy expenditure depends linearly on vehicle mass and distance; and (ii) formulating a mixed-integer nonlinear programming formulation and providing a novel relaxation to obtain a mixed-integer linear program. Extensive numerical results show that optimal solutions for most instances with up to 50 targets are obtained within one minute. For unsolved instances within a 10-minute limit, optimality gaps are under $13\%$, highlighting the formulation's tightness. We further benchmark our approach against three proposed baseline formulations and another algorithm from a related problem, and demonstrate that our formulation outperforms all baselines.
\end{abstract}

\section{Introduction}



The Traveling Salesman Problem (TSP) is a fundamental problem in combinatorial optimization, aiming to find the shortest tour that starts and ends at a depot while visiting a set of targets. TSP and its variants have broad applications, including unmanned aerial vehicle (UAV) routing. Although it is NP-hard, many formulations and generalizations have enabled efficient computational solutions.

This paper addresses the Load-Dependent Traveling Salesman Problem (LD-TSP), a variant of the classic TSP. In the LD-TSP, the cost of traveling between two locations (e.g., energy or fuel consumption) is not fixed; instead, it depends on the vehicle's current load. As the load decreases with each delivery, edge costs vary with the visitation sequence, making LD-TSP more complex than traditional TSP. 
Figure~\ref{fig: example_tour_package} illustrates a typical LD-TSP solution and how the vehicle mass varies as packages are delivered.

\begin{figure}[htb!]
    \centering
    \includegraphics[width = 0.8\linewidth]{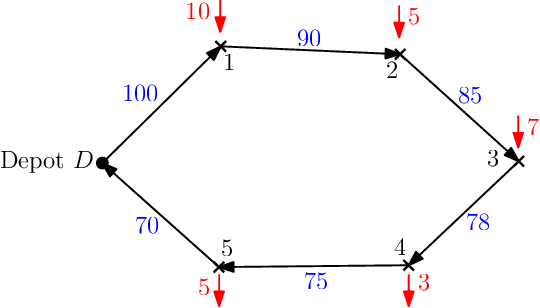}
    \caption{Example tour depicting dropping off packages and the mass of the vehicle after dropping each package}
    \label{fig: example_tour_package}
\end{figure}

LD-TSP has extensive applications, including routing cargo drones (i.e., UAV) or ground vehicles for package delivery. For UAVs, the carried load substantially affects fuel consumption and flight time. For example, a BlueHalo IE-V3 UAV with a 12~Ah battery sees its flight time decrease from 25 minutes without a load to under 15 minutes with a 5~lb package~\cite{bluehalo}. Accurate energy consumption modeling that accounts for changing load is thus essential for effective routing.

In~\cite{branch_cut_price_load_tsp}, a branch-cut-and-price algorithm for routing drones and trucks uses a non-linear energy model from~\cite{Dorling}. However, the model in~\cite{Dorling} overestimates flight power consumption for vehicle stages such as hovering, making it less accurate. In~\cite{Torabbeigi2020}, a two-stage approach selects depot locations and assigns drones using a linear payload-dependent power model derived using experiments. However, the focus is on minimizing drone usage rather than route optimization, leading to a mixed-integer linear program. In~\cite{drone_delivery_paths}, depot selection and drone scheduling are considered together using a non-linear power model that incorporates capacity and flight-time constraints.

While previous studies have focused on load-dependent vehicle routing with capacity constraints, primarily for drones with short flight times, the LD-TSP variant remains less explored. LD-TSP is especially relevant for vehicles with longer flight durations, including those listed in~\cite{drones_flight_time}, as well as ground vehicles. Load-dependent routing was examined in~\cite{MENDEZDIAZ20083223} for a Hamiltonian path problem with homogeneous package delivery. An integer program was solved via branch and bound, and additional valid inequalities were proposed to strengthen the linear relaxation lower bound. 
A related work considers a load-dependent vehicle routing problem, which minimizes the total distance traveled multiplied by vehicle gross weight~\cite{zachariadis2015load}. In contrast, we address a single-vehicle LD-TSP with heterogeneous payloads, where edge traversal costs explicitly depend on the remaining carried load as deliveries are made along the tour.

The most relevant contribution is~\cite{load_sequence_risk}, which studies a load- and sequence-dependent risk TSP for transporting hazardous cargo (hazmat TSP). Here, the edge cost is linearly proportional to the on-board load. An $A^*$-based tree search algorithm is used for optimal solutions, relying on a zero-cost return to the depot after all cargo is delivered. However, this assumption does not apply to our LD-TSP, as the vehicle continues to expend energy even after all packages are delivered and the final edge cost is not zero. 
In fact, our LD-TSP generalizes the hazmat TSP; by setting the vehicle's unladen mass (i.e., mass of vehicle without packages) to zero, we can mimic the zero-cost return scenario.
From our literature survey, we can observe that existing studies consider load-dependent vehicle routing or path problems; however, the LD-TSP itself remains unaddressed. Moreover, a rigorous derivation linking energy consumption to payload is lacking. In this work, we (i) propose a novel energy consumption model and show it depends linearly on the carried load, (ii) formulate the LD-TSP as a mixed-integer nonlinear program (MINLP) and introduce a relaxation to obtain a mixed-integer linear program (MILP), which we solve using Gurobi, and (iii) present extensive computational results, solving instances with up to 50 targets to optimality in under one minute and, achieve solutions within $13\%$ of optimality for instances with 70 targets. (iv) We benchmark our formulation against three baseline formulations: two MILP variations, one MINLP. Additionally, we benchmark against the algorithm from~\cite{load_sequence_risk} for the hazmat TSP. We demonstrate that our approach achieves better solution quality or shorter computation time compared to all baselines.

\section{Problem Formulation}

The key challenge in LD-TSP is accurately determining the energy required to travel between locations. We first present an energy dissipation model for a vehicle operating in a constant wind field.\footnotemark\footnotetext{Although wind effects are not considered in our LD-TSP formulation, we include wind in the energy dissipation model for generality.} We then formulate the optimal control problem to minimize energy, and show that this is equivalent to minimizing time or distance. This equivalence also holds in the absence of wind, which is the case for our LD-TSP routing study.

\subsection{Energy dissipation model}
Consider an agent moving through a medium with constant speed \(V_0\) relative to the medium. The power dissipation \(P\) required from the on-board energy source is given by the dot product of the dissipative force \({\bf F}_d\) and the agent’s velocity \({\bf v}\). The dissipative force acts opposite to the longitudinal direction of motion and has magnitude \(\phi(V_0)\), which depends solely on \(V_0\). The agent’s velocity comprises its speed and the medium’s velocity \({\bf v}_w\). The component of power dissipation due to the agent's motion is \(\phi(V_0)V_0\). When the medium is moving, an additional component arises that depends on the heading angle \(\theta\) between the agent and the medium, given by \(\phi(V_0)v_w\cos\theta\). Consequently, the total dissipated power can be expressed as
\begin{align} \label{eq: power}
    P(\theta) = P_0 + P_1\cos\theta,
\end{align}
where $P_0 = \phi(V_0)V_0,\ P_1 = \phi(V_0)v_w$.
The assumption \(v_w < V_0\) ensures that the agent can overcome the medium’s influence, maintain controllability, and avoid prohibitive energy costs for upstream motion, thereby guaranteeing the feasibility of navigation. We provide a detailed derivation of the power model and list the assumptions in the Appendix.

Since every segment of the route must be energy-efficient, we pose the following motion planning for the vehicle to efficiently transport packages from the $i^{th}$  to the $j^{th}$ target. 
\begin{align}
C_{ij} := \min \quad \int_0^T (P_0 + P_1 \cos \theta(t)) \; dt, 
\end{align}
subject to the vehicle's kinematic constraints, given by 
\begin{align}
\label{eq:kinematics}
\dot x (t) = V_0 \cos \theta (t) + v_w, \,\, \dot y (t) = V_0 \sin \theta (t), \,\, \dot \theta (t) = u (t), 
\end{align}
and the following specified boundary conditions:
\begin{align} 
\label{eq:BC1}
(x(0), y(0), \theta(0)) = (x_i, y_i, \theta_i), \\ 
\label{eq:BC2} (x(T), y(T), \theta(T)) = (x_j, y_j, \theta_j). \end{align}
Here, $x, y$ denote the 2D coordinates {in the global frame}, and $u$ controls the rate of orientation change.

{\bf Remark:}
Note that the integrand in the objective function represents power consumption; thus, the objective is to minimize total energy while satisfying kinematic constraints and boundary conditions. Here, $T$ is treated as a free variable.

\subsection{Equivalence of time optimal and energy optimal paths}

We prove the equivalence of minimizing energy and distance in the following theorem.
\begin{theorem} \label{theorem}
    A minimum energy path is equivalent to a minimum time path.
\end{theorem}
\begin{proof}
    Let \begin{align}
c_{ij} = \quad \min \; \int_0^T dt, 
\end{align} 
subject to kinematic constraints given by \eqref{eq:kinematics} and the following specified boundary conditions \eqref{eq:BC1} and \eqref{eq:BC2}.
Noting that $\cos \theta(t) = \frac{\dot x (t) - v_w}{V_0},$ and noting that $v_w, V_0$ are constants, we can see that 
\begin{align}
\begin{split}
    C_{ij}^* &=  \min \; \int_0^T \left(P_0 + P_1  \frac{\dot x (t) - v_w}{V_0} \right) \; dt \\
    &= \int_0^T \left(P_0 -P_1 \frac{v_w}{V_0} \right) \; dt + \frac{P_1}{V_0} \left(x_j-x_i \right).
\end{split}
\end{align}
Since $\left(P_0 -P_1 \frac{v_w}{V_0} \right) >0$ and $x_j - x_i$ depends only on the boundary conditions, 
\begin{align*}
    C_{ij}^* = \left(P_0 - P_1 \frac{v_w}{V_0} \right) c_{ij} + \frac{P_1}{V_0} (x_j - x_i).
\end{align*}
As the second term is a constant and $c_{ij}$ is the only varying component of the first term, which represents time, minimizing time clearly minimizes the energy expended. Hence, time-optimal and energy-optimal paths are equivalent.
\end{proof}

Let $M$ be the nominal/unladen mass of the vehicle (when it returns to the depot), and let $M_i$ represent its mass at the $i^{th}$ target after delivering a package of mass $m_i$ to the $i^{th}$ customer. 
Let $C_{ij}$ correspond to the energy cost associated with the vehicle having a nominal mass $M$. Since the lift needed by the drone is directly proportional to its load, we may represent the cost of travel when the mass is $M_i$ to be $\frac{M_i}{M}C_{ij}$. 

\textbf{Remark:} In a {\it still} medium, $v_w = 0$, and the term $P_1 = 0.$ Hence, the energy dissipated, given by $\frac{M_i}{M} C_{ij},$ can be equivalently represented in terms of the distance between the locations $i$ and $j,$ denoted by $d_{ij},$ as $\alpha M_i d_{ij}.$ Hence, $\alpha$ is a constant dependent on the drag coefficient $C_d$ and frontal area $A_f,$ which are influenced by the dimensions of the vehicle, and the nominal/unladen mass $M$ of the vehicle. Therefore, $\alpha$ is a constant.

\section{Nonlinear Program for LD-TSP}

We begin by defining the notation used in our mathematical formulation.
Let $T$ be the set of target locations, and $m_t$ the mass to be dropped at $t \in T$. Let $D$ denote the depot, where the vehicle begins after picking up all packages destined for $T$. Define $E$ as the set of edges $(i, j)$ with $i, j \in T \bigcup \{D\}$; traversing edge $(i, j)$ incurs energy consumption $C_{ij}$ and distance $d_{ij}$. For any $S \subset T \cup \{D\}$, $S \neq \emptyset$, let $\delta^+(S) = \{(i, j): i \in S,\, j \notin S\}$ denote outgoing edges from $S$, and $\delta^-(S) = \{(i, j): i \notin S,\, j \in S\}$ the incoming edges to $S$. Let $M_t$ be the vehicle’s mass after dropping $m_t$ at location $t \in T$, and $M$ its unladen mass (i.e., without packages). Set $\overline{M} = \sum_{t \in T} m_t$ as the total package mass, so that the initial mass at the depot is $M_D = M + \overline{M}$. Finally, $x_{ij}$ is a binary variable indicating whether edge $(i, j)$ is traversed ($x_{ij} = 1$) or not ($x_{ij} = 0$).\footnotemark
\footnotetext{Throughout this paper, we use $x_{ij}$ and $x_e$ interchangeably to represent the binary variable for edge $e = (i, j)$.}


The objective is to minimize the total energy expenditure, expressed as $\min \sum_{e \in E} C_e x_e$. Since minimizing energy is equivalent to minimizing distance, we set $C_e = \alpha M_e d_e$, so the objective becomes
\begin{align} \label{eq: modified_objective_function}
    \min \sum_{(i, j) \in E} \alpha M_i d_{ij} x_{ij}.
\end{align}
There are two main sets of constraints. First, degree constraints ensure the depot and each target have exactly one incoming and one outgoing edge, given by
\begin{align}
    \sum_{e \in \delta^+({t})} x_e &= 1, \quad \forall t \in T \cup \{D\}, \label{eq: TSP_constraint_1} \\
    \sum_{e \in \delta^-({t})} x_e &= 1, \quad \forall t \in T \cup \{D\}. \label{eq: TSP_constraint_2}
\end{align}
The second set of constraints ensures that the desired package with mass $m_t$ is dropped at each location $t$. Specifically,
\begin{align} \label{eq: mass_dropping_constraint}
    |M_i - M_j - m_j x_{ij}| \leq \overline{M} (1 - x_{ij}), \, \forall i \in T \cup \{D\},\, j \in T.
\end{align}
This constraint relates the vehicle’s mass before and after visiting location $j$ via edge $(i, j)$. If $x_{ij} = 1$, then $M_i - M_j = m_j$, indicating that a package of mass $m_j$ is dropped at $j$. A depiction of the same is shown in Fig.~\ref{fig: Mij_relation}. For $x_{ij} = 0$, the mass difference is upper-bounded by $\overline{M}$ for any route. Note that $j = D$ is excluded since no package is dropped at the depot.

\begin{figure}[htb!]
    \centering
    \includegraphics[width=0.5\linewidth]{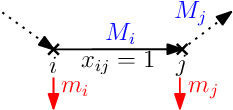}
    \caption{Depiction of relation between $M_i$ and $M_j$}
    \label{fig: Mij_relation}
\end{figure}




Note that~\eqref{eq: mass_dropping_constraint} also eliminates subtours, as it resembles the Miller-Tucker-Zemlin subtour elimination constraints for TSP~\cite{MTZ}.

Finally, the decision variables and their bounds are:
\begin{align} \label{eq: decision_variables}
\begin{split}
    x_{ij} &\in \{0, 1\}, \quad \forall (i, j) \in E, \\
    M_i &\in \left[M, M + \overline{M} \right], \quad \forall i \in T, \\
    M_D &= M + \overline{M}.
\end{split}
\end{align}
This formulation is a MINLP due to bilinear terms in the objective in \eqref{eq: modified_objective_function}. In the following section, we present a relaxation to obtain an MILP.

\section{Mixed Integer Linear Program for LD-TSP}

To relax the original MINLP, we replace the decision variables $M_i$ for $i \in T \cup \{D\}$ with auxiliary variables defined on each edge $(i, j) \in E$:
\begin{align} \label{eq: auxilliary_variables}
    \zeta_{ij} = M_i x_{ij}, \quad \eta_{ij} = M_j x_{ij}.
\end{align}
With these variables, the objective in~\eqref{eq: modified_objective_function} becomes
\begin{align} \label{eq: modified_objective_function_vars}
    \min \sum_{(i, j) \in E} \alpha \zeta_{ij} d_{ij}.
\end{align}
From~\eqref{eq: auxilliary_variables}, it follows that
\begin{align} \label{eq: auxilliary_variables_depot}
    \zeta_{Dj} = M_D x_{Dj}, \quad \eta_{jD} = M x_{jD}, \quad \forall j \in T,
\end{align}
since the vehicle leaves the depot with mass $M_D$ (laden) and returns with mass $M$ (unladen).


We now derive constraints for $\zeta_{ij}$ and $\eta_{ij}$. Multiplying~\eqref{eq: mass_dropping_constraint} by $x_{ij}$ and using the fact that $x_{ij}^2 = x_{ij}$ (since $x_{ij}$ is binary), we obtain
\begin{align} \label{eq: zeta_eta_relation}
    \zeta_{ij} - \eta_{ij} - m_j x_{ij} = 0, \quad \forall i \in T \cup \{D\},\, j \in T,
\end{align}
where $\overline{M} x_{ij} (1 - x_{ij})$ reduces to zero. This relation connects $\zeta_{ij}$ and $\eta_{ij}$ on each edge.  
Additionally, to relate $\zeta$ and $\eta$ at a common target $j$ (see Fig.~\ref{fig: zeta_eta_relation}), note that $\eta_{ij} = M_j x_{ij}$ and $\zeta_{jk} = M_j x_{jk}$; thus,
\begin{align} \label{eq: zeta_eta_relation_2}
    \sum_{i \in \delta^{-}(j)} \eta_{ij} = \sum_{k \in \delta^{+}(j)} \zeta_{jk}, \quad \forall j \in T,
\end{align}
since both sides equal $M_j$, by~\eqref{eq: TSP_constraint_1} and~\eqref{eq: TSP_constraint_2}. The sums are taken over all incoming and outgoing edges, respectively, as the target from which we arrive at $j$ (i.e., $i$), and that target to which we depart from $j$ (i.e., $k$) is not known.
\begin{figure}[htb!]
    \centering
    \includegraphics[width=0.7\linewidth]{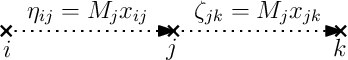}
    \caption{Depiction of $\eta_{ij}$ and $\zeta_{jk}$ to derive their relationship}
    \label{fig: zeta_eta_relation}
\end{figure}

Additionally, noting that the mass of the vehicle on each edge ($M_i$) can be bounded by
\begin{align} \label{eq: bounds_Mi}
    0 \leq M_i - M \leq \overline{M},
\end{align}
and noting that 
\begin{align}
    0 \leq x_{ij} \leq 1, \quad \forall (i, j) \in E,
\end{align}
the following inequalities can be obtained for every $(i, j) \in E$:
\begin{align}
    &\left(M_i - M \right) x_{ij} \geq 0, \label{eq: intermediary_eq_1} \\
    &\left(M + \overline{M} - M_i \right) x_{ij} \geq 0, \label{eq: intermediary_eq_2} \\
    &\left(M_i - M \right) \left(1 - x_{ij} \right) \geq 0, \label{eq: intermediary_eq_3} \\
    &\left(M + \overline{M} - M_i \right) \left(1 - x_{ij} \right) \geq 0. \label{eq: intermediary_eq_4}
\end{align}
We aim to rewrite these equations using the definition of $\zeta_{ij}$ from~\eqref{eq: auxilliary_variables}. Since we do not explicitly track the vehicle’s sequence, $M_i$ must be eliminated in some equations. This can be done by substituting the definition of $\zeta_{ij}$ and using the degree constraints in~\eqref{eq: TSP_constraint_1} as follows:
\begin{align} \label{eq: eliminate_Mi}
    \sum_{j \in \delta^+(\{i\})} \zeta_{ij} = \sum_{j \in \delta^+(\{i\})} M_i x_{ij} = M_i.
\end{align}

Consider~\eqref{eq: intermediary_eq_1} and~\eqref{eq: intermediary_eq_2}; these can be rewritten as\footnotemark\footnotetext{When $x_{ij} = 0,$ i.e., the edge is not traversed, $\zeta_{ij}$ is set to zero with the constraints in \eqref{eq: auxilliary_var_constraint_1}. Therefore, the linear relaxation is tight.}
\begin{align} \label{eq: auxilliary_var_constraint_1}
    M x_{ij} \leq \zeta_{ij} \leq \left(M + \overline{M} \right) x_{ij}, \quad \forall (i, j) \in E.
\end{align}
Consider \eqref{eq: intermediary_eq_3}; this equation can be rewritten using the definition of $\zeta_{ij}$ in \eqref{eq: auxilliary_variables} and eliminating $M_i$ using \eqref{eq: eliminate_Mi} as
\begin{align} \label{eq: zeta_var_constraint_2}
    \zeta_{ij} \leq \sum_{j \in \delta^+(\{i\})} \zeta_{ij} - M + M x_{ij}, \quad \forall (i, j) \in E.
\end{align}
Similarly, \eqref{eq: intermediary_eq_4} can be rewritten as
\begin{align} \label{eq: zeta_var_constraint_3}
    \zeta_{ij} \geq \sum_{j \in \delta^+(\{i\})} \zeta_{ij} + \left(M + \overline{M} \right) \left(x_{ij} - 1 \right), \quad \forall (i, j) \in E.
\end{align}

Similarly, for $\eta_{ij}$, we use the inequality $0 \leq M_j - M \leq \overline{M}$, instead of~\eqref{eq: bounds_Mi}, to derive analogous constraints to~\eqref{eq: intermediary_eq_1}--\eqref{eq: intermediary_eq_4}. The resulting equations for $\eta_{ij}$ are:
\begin{align}
    &M x_{ij} \leq \eta_{ij} \leq \left(M + \overline{M} \right) x_{ij}, \quad \forall (i, j) \in E, \label{eq: eta_var_constraint_1} \\
    &\eta_{ij} \leq \sum_{i \in \delta^+(\{j\})} \eta_{ij} - M + M x_{ij}, \quad \forall (i, j) \in E, \label{eq: eta_var_constraint_2} \\
    &\eta_{ij} \geq \sum_{i \in \delta^+(\{j\})} \eta_{ij} + \left(M + \overline{M} \right) \left(x_{ij} - 1 \right), \quad \forall (i, j) \in E. \label{eq: auxilliary_variables_final_constraint}
\end{align}

\textbf{Remark:} The constraints in~\eqref{eq: zeta_var_constraint_2}, \eqref{eq: zeta_var_constraint_3}, \eqref{eq: eta_var_constraint_2}, and \eqref{eq: auxilliary_variables_final_constraint} can be omitted from the MILP, as they do not provide meaningful restrictions for $x_{ij} = 0$ or $1$. In fact, we consider an alternative MILP formulation including these constraints, and observe from computational experiments that their inclusion degrades algorithm performance. 

\subsection{Summary of MILP formulation}

Our final formulation consists of the objective in~\eqref{eq: modified_objective_function_vars} and the constraints~\eqref{eq: TSP_constraint_1}, \eqref{eq: TSP_constraint_2}, \eqref{eq: auxilliary_variables_depot}--\eqref{eq: zeta_eta_relation_2}, \eqref{eq: auxilliary_var_constraint_1}, and \eqref{eq: eta_var_constraint_1}. Since $\zeta_{ij}$ and $\eta_{ij}$ are continuous decision variables and $x_{ij}$ are binary, the resulting program is an MILP.

\section{Alternate MILP formulations and baselines}

We benchmark our formulation against three alternative baseline formulations and an additional algorithm. The first baseline is identical to the previously defined MILP, but additionally includes the constraints~\eqref{eq: zeta_var_constraint_2}, \eqref{eq: zeta_var_constraint_3}, \eqref{eq: eta_var_constraint_2}, and \eqref{eq: auxilliary_variables_final_constraint} to assess their effect on the lower bound. The second baseline augments this baseline MILP with subtour elimination constraints from the Dantzig-Fulkerson-Johnson (DFJ) formulation for the TSP~\cite{TSP_book_applegate}:
\begin{align}
    \sum_{e \in \delta^+(S)} x_e \geq 1, \quad \forall S \subset T \cup \{D\}, \, |S| \geq 1. \label{eq: TSP_constraint_3}
\end{align}
Due to the exponential number of such constraints, we implement them as lazy cuts: violated constraints are identified using a separation algorithm and added to the solver. The separation algorithm is based on a flow constraint~\cite{3_approx_algo_2_vehicle_TSP}, requiring that the flow from the depot to each target be at least one. For any set $S$ containing the depot $D$, this is
\begin{align*}
    \sum_{(i, j) \in \delta^+ (S)} x_{ij} \geq 1, \quad 1 \leq |S| < |T \cup \{D\}|, \quad D \in S.
\end{align*}
Since the possible sets $S$ grow exponentially, we search for violated sets using the max-flow min-cut theorem after each relaxed solution, and add only these constraints to the solver.


Our third baseline is the previously presented MINLP formulation for the LD-TSP, with the objective in~\eqref{eq: modified_objective_function} and constraints~\eqref{eq: TSP_constraint_1}--\eqref{eq: mass_dropping_constraint}. Here, $x_{ij}$ are binary and $M_i$ are continuous decision variables, with bounds detailed in~\eqref{eq: decision_variables}. 

Additionally, we benchmark against an $A^*$-based tree search algorithm from~\cite{load_sequence_risk} for the hazmat TSP. 

\section{Results} \label{sect: results}

The MILP formulation, its variants, and the MINLP were implemented in Julia~1.11 using the JuMP package on a laptop with an AMD Ryzen 9 5900HS CPU (3.30~GHz, 16~GB RAM). MILPs without DFJ subtour cuts and the MINLP were solved via branch-and-bound, while the MILP with DFJ cuts employed branch-and-cut; all utilized the commercial solver Gurobi~12.0.2~\cite{gurobi}. Each instance was warm-started with two visitation sequences, generated by solving the symmetric TSP (with Euclidean distances) using LKH~\cite{LKH_TSP_solver}, a state-of-the-art heuristic.\footnotemark\, The sequence yielding the lower LD-TSP cost was used for initialization.
\footnotetext{Although the two TSP tours, such as $D-a-b-c-d-e-D$ and $D-e-d-c-b-a-D$, are equivalent for TSP, they yield different costs for LD-TSP; hence, we consider both sequences.}

We consider nine instances with node counts ranging from $11$ to $70$. The 11-node instance is taken from~\cite{MD_algorithm}, referred to as MM1 or MS1\footnotemark, while the remaining instances are standard TSP cases from~\cite{TSP_ref_1} and~\cite{TSP_ref_2}. For all instances, the last node is designated as the vehicle depot by default.
\footnotetext{In~\cite{wang}, the MM1 instance involves 10 targets and three vehicles departing from different locations. For our problem, we only use one depot and ignore the second and third departure points. Detailed coordinates for targets and depots are given in~\cite{wang}.}

Without loss of generality, we set the maximum package mass to $1$ unit, making the unladen UAV mass $M$ the key parameter. We vary the unladen mass factor $\gamma \in \{2, 5, 10\}$, where $\gamma := M/\overline{M}$. Package masses at each location are randomly selected from $\{0.1, 0.2, \ldots, 0.9, 1.0\}$. Additionally, we set $\alpha = 0.1$ without loss of generality.

We report solutions obtained for solver time limits of 1, 5, and 10 minutes. In the following subsections, we analyze factors influencing the results, including the impact of unladen mass on the optimal solution and differences compared to the standard TSP. Additionally, we benchmark our initial MILP formulation against three baseline formulations. Moreover, we compare our MILP results to the solutions against the algorithm from~\cite{load_sequence_risk} for the hazmat TSP. 

\textbf{Remark:} 
We report solutions for three solver time limits, reflecting practical scenarios where users may require quick, high-quality results within a limited timeframe.

\textbf{Remark:} 
Without loss of generality, we set the maximum package mass to $1$ unit; for other instances, package and vehicle masses can be scaled accordingly. Additionally, we fix $\alpha = 0.1$, as it affects only the objective value and not the solution itself.

\subsection{Optimality gap for various instances and computation time for unladen factor $\gamma = 10$}

In Fig.~\ref{fig: results_unladen_factor_10}, we present the optimality gap\footnotemark\, for the various instances considered in this paper for $\gamma = 10$. 
Figure~\ref{subfig: all_instances} shows the optimality gap for a $1$-minute time limit. Our formulation finds optimal solutions for all instances except st70, with smaller instances solved in seconds. For st70, we achieve a solution within $3\%$ of optimality in $1$ minute, highlighting the tightness of our approach.
\footnotetext{The optimality gap indicates how close the best feasible solution is to the optimum, relative to the best lower bound from the linear relaxation. Smaller gaps mean solutions closer to optimality.}

\begin{figure}[htb!]
    \centering
    \subfloat[Optimality gap and computation time for all instances]{\includegraphics[width = 0.7\linewidth]{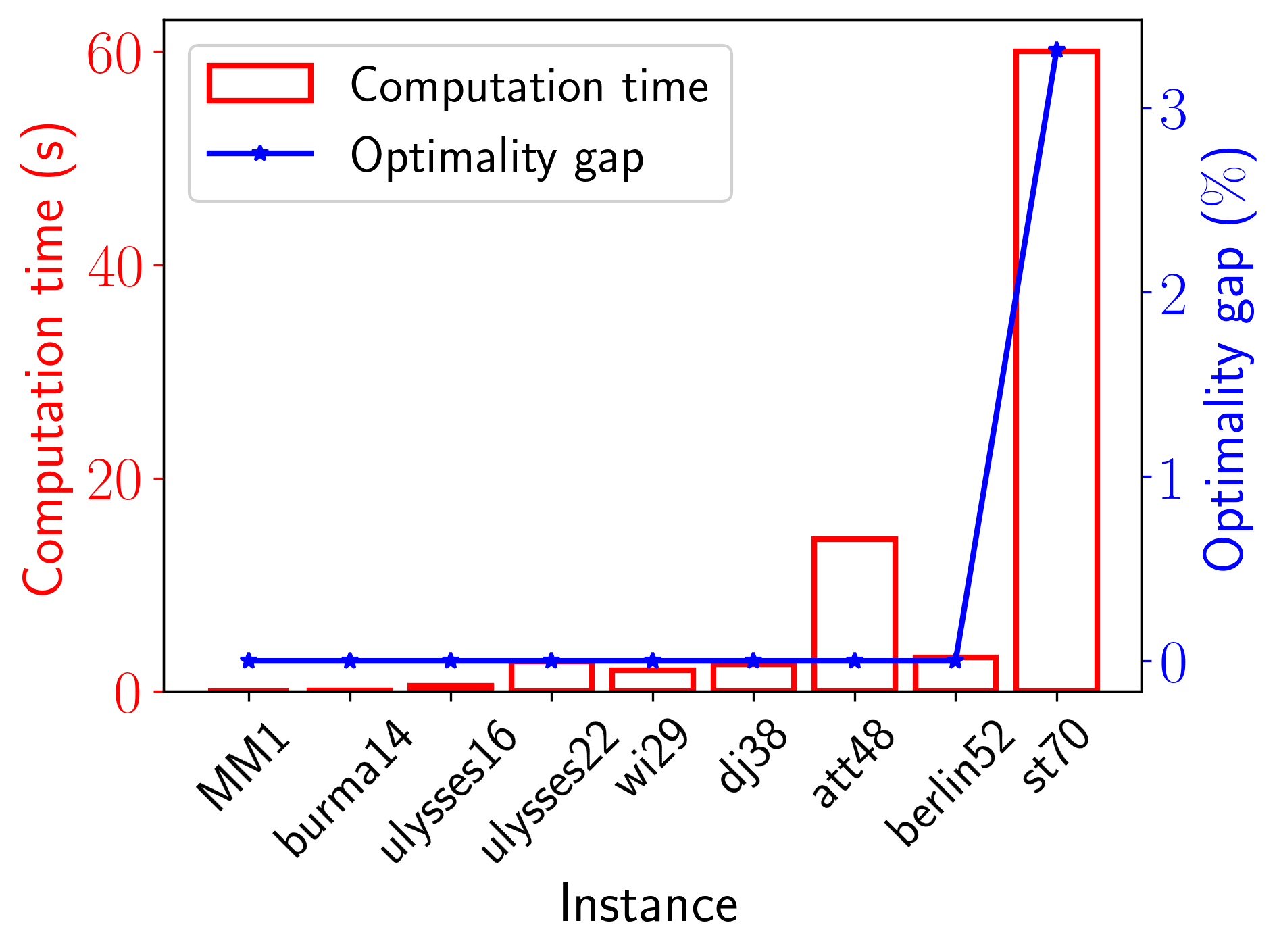}
    \label{subfig: all_instances}} \hfil
    \subfloat[st70 for varying time limits]{\includegraphics[width = 0.7\linewidth]{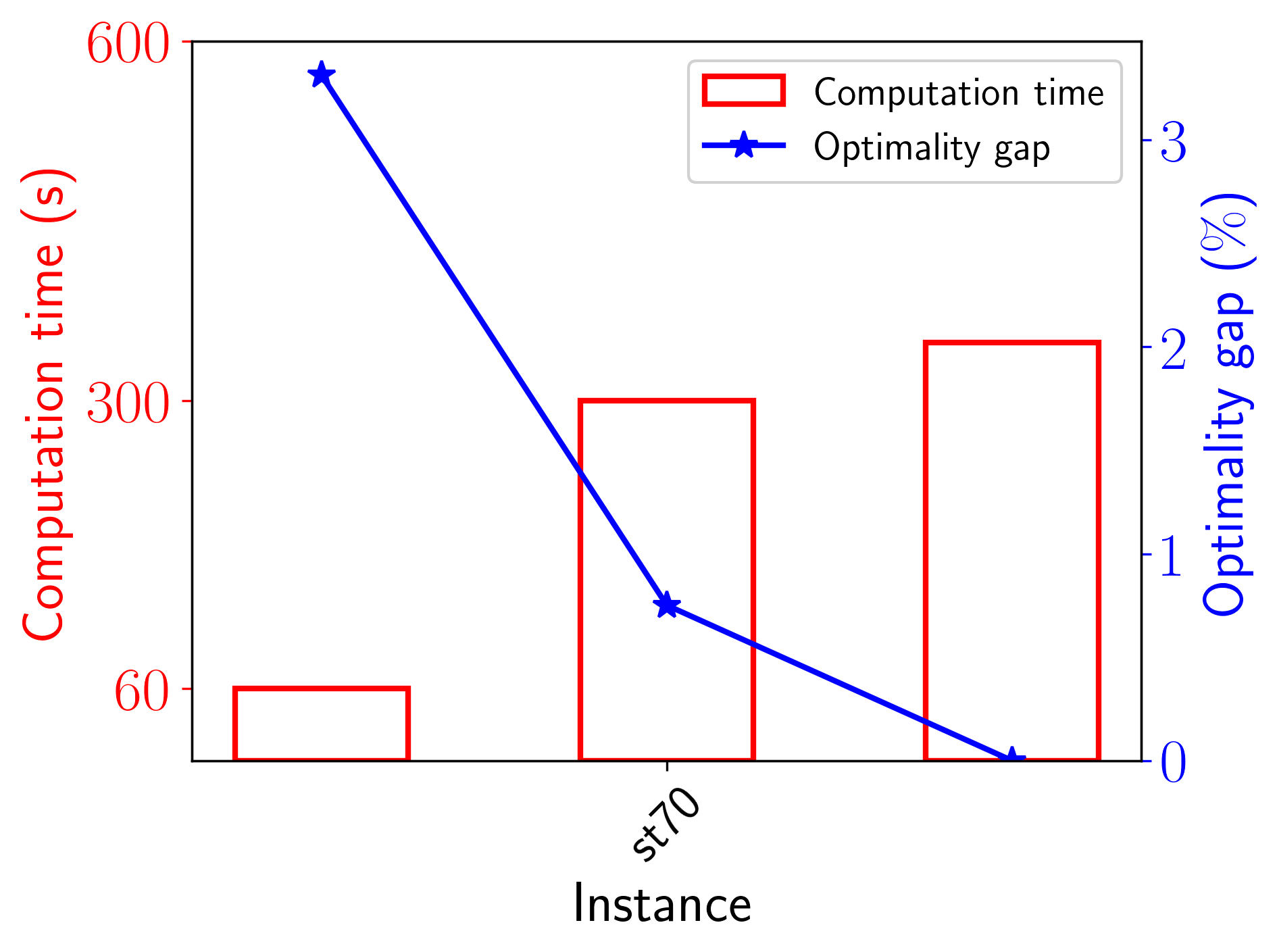}
    \label{subfig: st70}}
    \caption{Summary of optimality gap and computation time for an unladen mass ratio of $10$ for the vehicle}
    \label{fig: results_unladen_factor_10}
\end{figure}

We also show the optimality gap for st70 with solver time limits of $5$ and $10$ minutes in Fig.~\ref{subfig: st70}. The gap decreases to about $1\%$ after $5$ minutes and is closed completely with optimality achieved in around $6$ minutes. Thus, longer computation times allow users to obtain tighter optimality gaps in practice, as expected.

\begin{figure*}[htb!]
    \centering
    \subfloat[Optimality gap and computation time  first three instances]{\includegraphics[width = 0.65\linewidth]{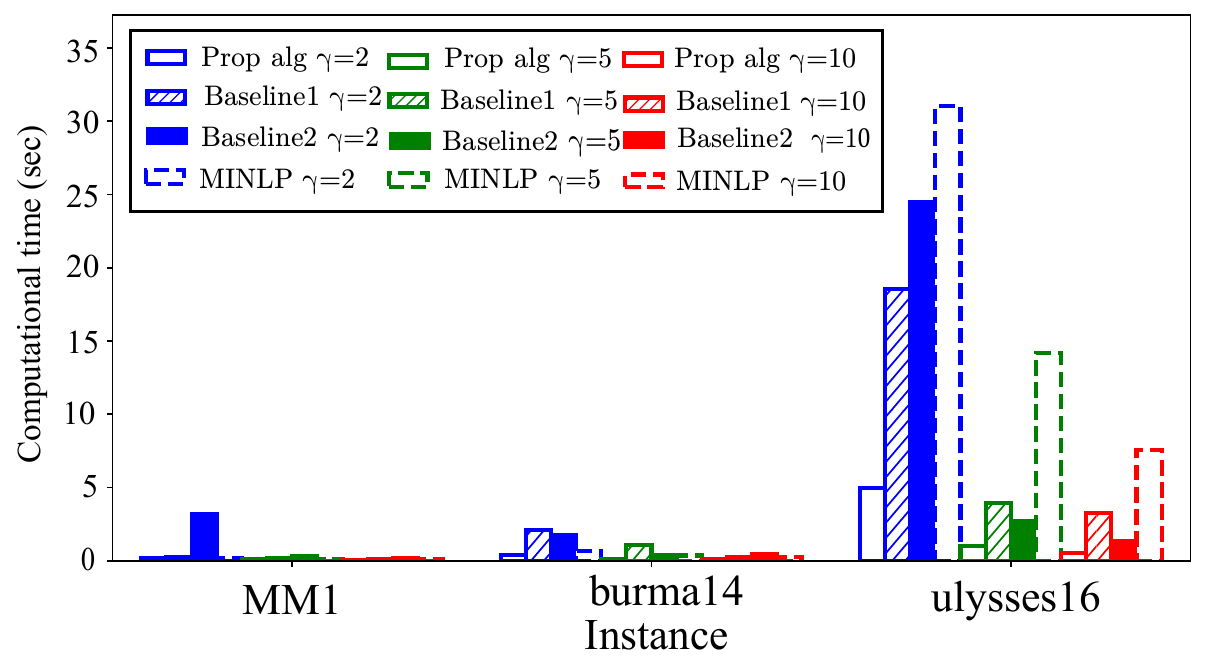}}
    \label{fig:image1}
    \subfloat[Optimality gap and computation time for other six instances]{\includegraphics[width = 0.8\linewidth]{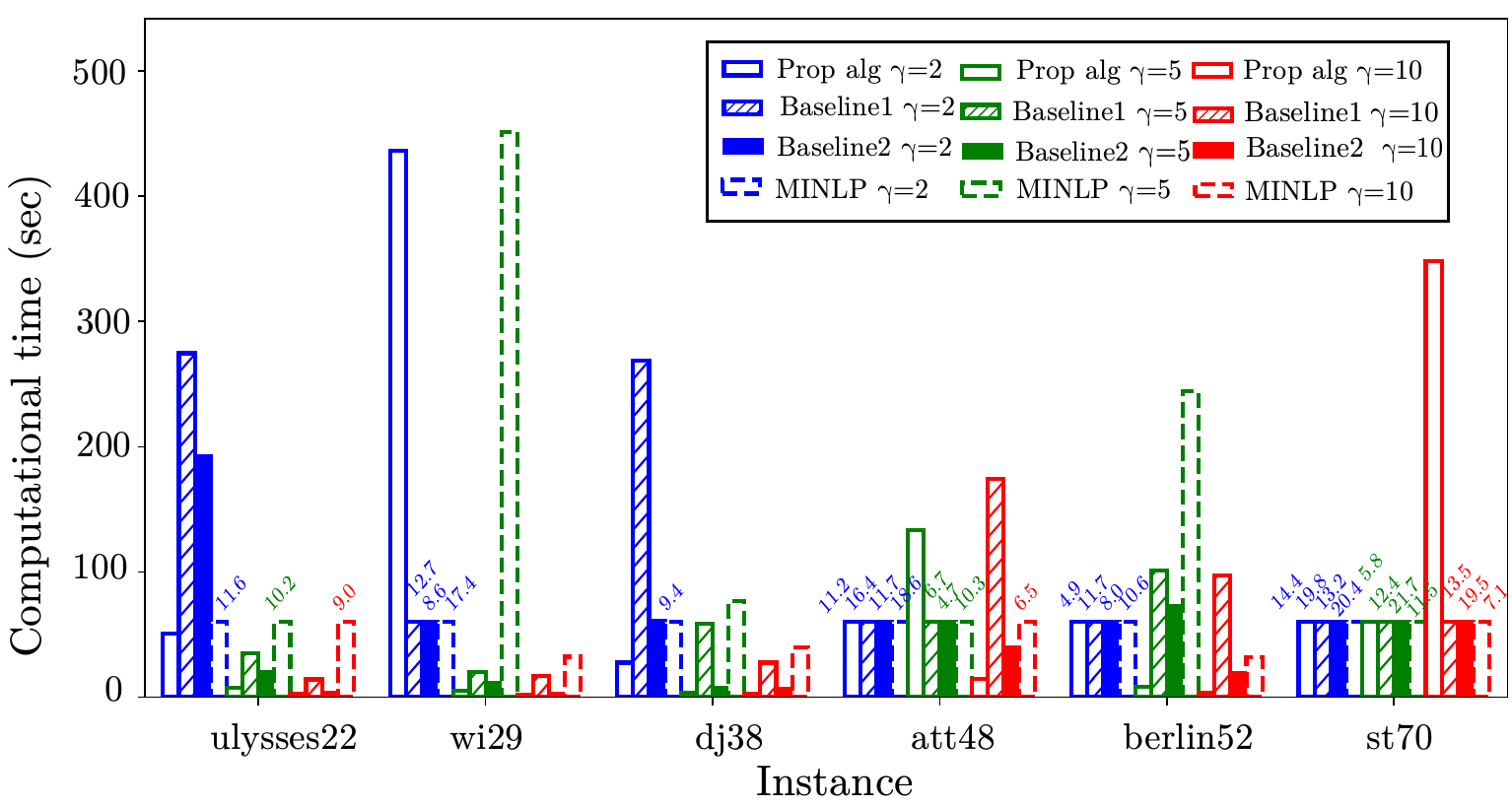}
    \label{fig:image2}}
    \caption{Computation time with varying unladen mass factor $\gamma,$ and optimality gap (mentioned above the bar graph) when optimal solution wasn't obtained within 10 minutes for proposed formulation and three benchmarking formulations. Refer to ``Alternate MILP formulations" section for the definitions of Baselines 1 and 2.}
    \label{fig: computation_time_changing}
\end{figure*}

\subsection{Impact of changing unladen vehicle mass for same packages}

We vary the unladen mass factor $\gamma$ from $2$ to $5$ to $10$ for the proposed MILP and three baseline formulations. Figure~\ref{fig: computation_time_changing} illustrates the computation times for different $\gamma$. Across all instances and methods, higher values of $\gamma$ lead to reduced computation times for finding the optimal solution. For instances such as st70, where optimality could not be reached within the time limit, a larger $\gamma$ corresponds to a smaller optimality gap at termination. This trend likely arises because, as $\gamma$ increases, the cost term, dominated by the unladen mass, lessens the impact of route sequencing. In Fig.~\ref{fig: computation_time_changing}, we also report the optimality gaps for instances not solved to optimality after $10$ minutes. In such cases, we report the solution obtained with a time limit of $1$ minute.

Figure~\ref{fig: computation_time_changing} shows that our proposed MILP consistently outperforms all three baselines in computation time and/or optimality gap across all instances. Notably, for cases like wi29 where the baselines fail to reach optimality within $10$ minutes (for $\gamma = 2$), our MILP finds the optimal solution in under $500$ seconds. This indicates that including additional constraints and/or DFJ subtour cuts reduces algorithm efficiency. Among the baselines, the MINLP performs worst, highlighting the effectiveness of our MILP relaxation.


Beyond computation time, we observe that the optimal tour can also vary with $\gamma$. Figure~\ref{fig: MM1_mass_ratio} shows the optimal tour for the MM1 instance with ten targets, where package masses to be dropped are fixed but the visitation sequence changes as $\gamma$ varies. This example highlights how two vehicles with different unladen masses may follow different delivery routes, even when transporting identical packages.

\begin{figure*}[htb!]
    \centering
    \subfloat[Unladen mass ratio of $2$]{\includegraphics[width = 0.32\textwidth]{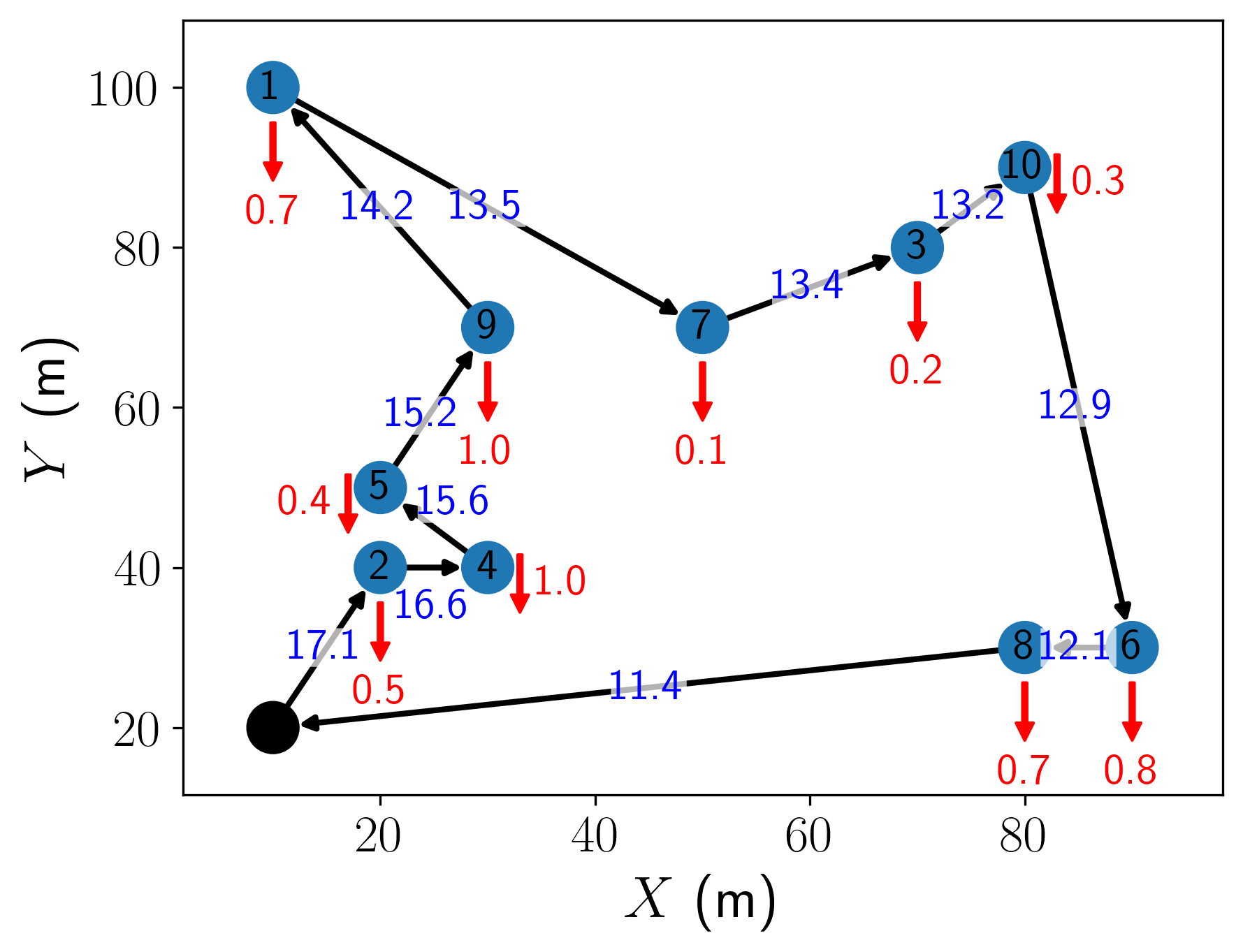}
    \label{subfig: MM1_factor_2}}
    \hfil
    \subfloat[Unladen mass ratio of $5$]{\includegraphics[width = 0.32\textwidth]{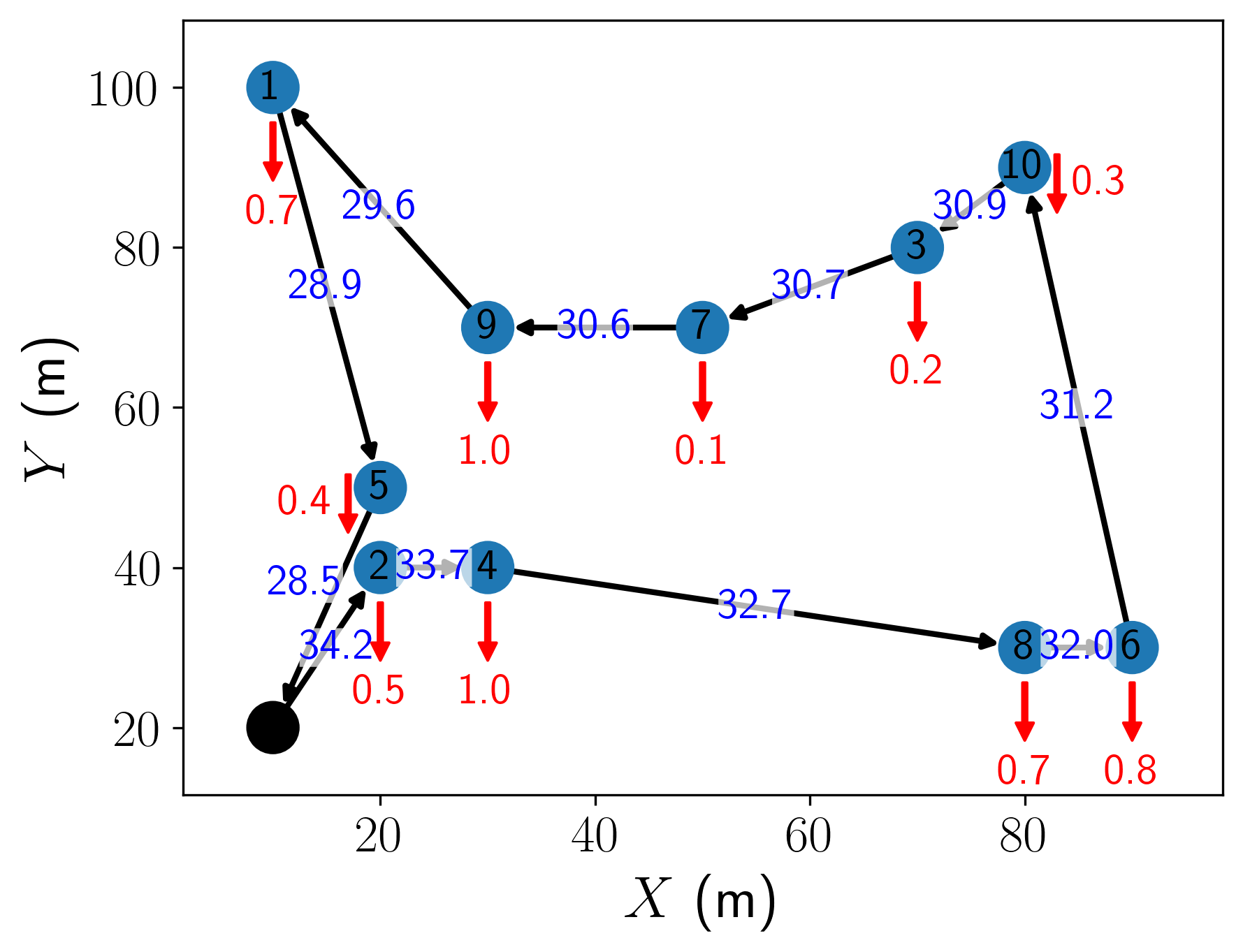}
    \label{subfig: MM1_factor_5}}
    \hfil
    \subfloat[Unladen mass ratio of $10$]{\includegraphics[width = 0.32\textwidth]{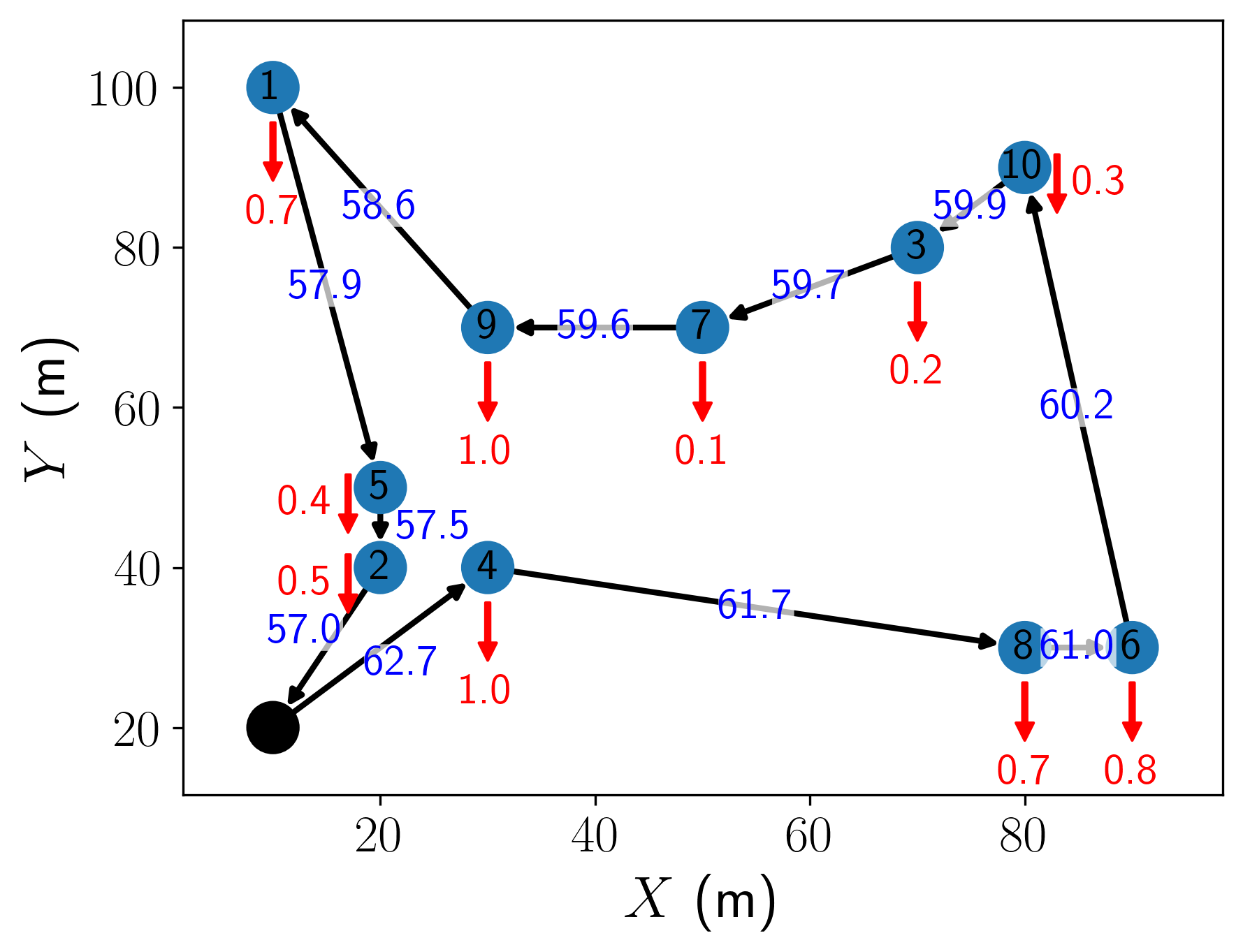}
    \label{subfig: MM1_factor_10}}
    \caption{Effect on unladen mass ratio on optimal path for instance MM1}
    \label{fig: MM1_mass_ratio}
\end{figure*}

\subsection{Differences in optimal TSP solution and optimal LD-TSP solution}

A simple approach to obtain a feasible solution is to solve the regular symmetric TSP and use its sequence to construct an LD-TSP solution; we, in fact, use this technique to warm start our solver. However, this approach does not always yield the optimal LD-TSP solution. For instance, using the TSP-derived sequence for the ulysses22 instance produces a solution with a cost $257.40$ (Fig.~\ref{subfig: optimal_tsp_tour}), while the optimal LD-TSP tour has a cost of $246.15$ (Fig.~\ref{subfig: optimal_tour}). Notably, the tour structures differ. In the optimal solution, targets 12 and 13 are visited early, whereas in the TSP sequence, they are visited later.

\begin{figure}[htb!]
    \centering
    \subfloat[Optimal TSP tour utilized for LD-TSP problem (cost = $257.40$)]{\includegraphics[width = 0.45\textwidth]{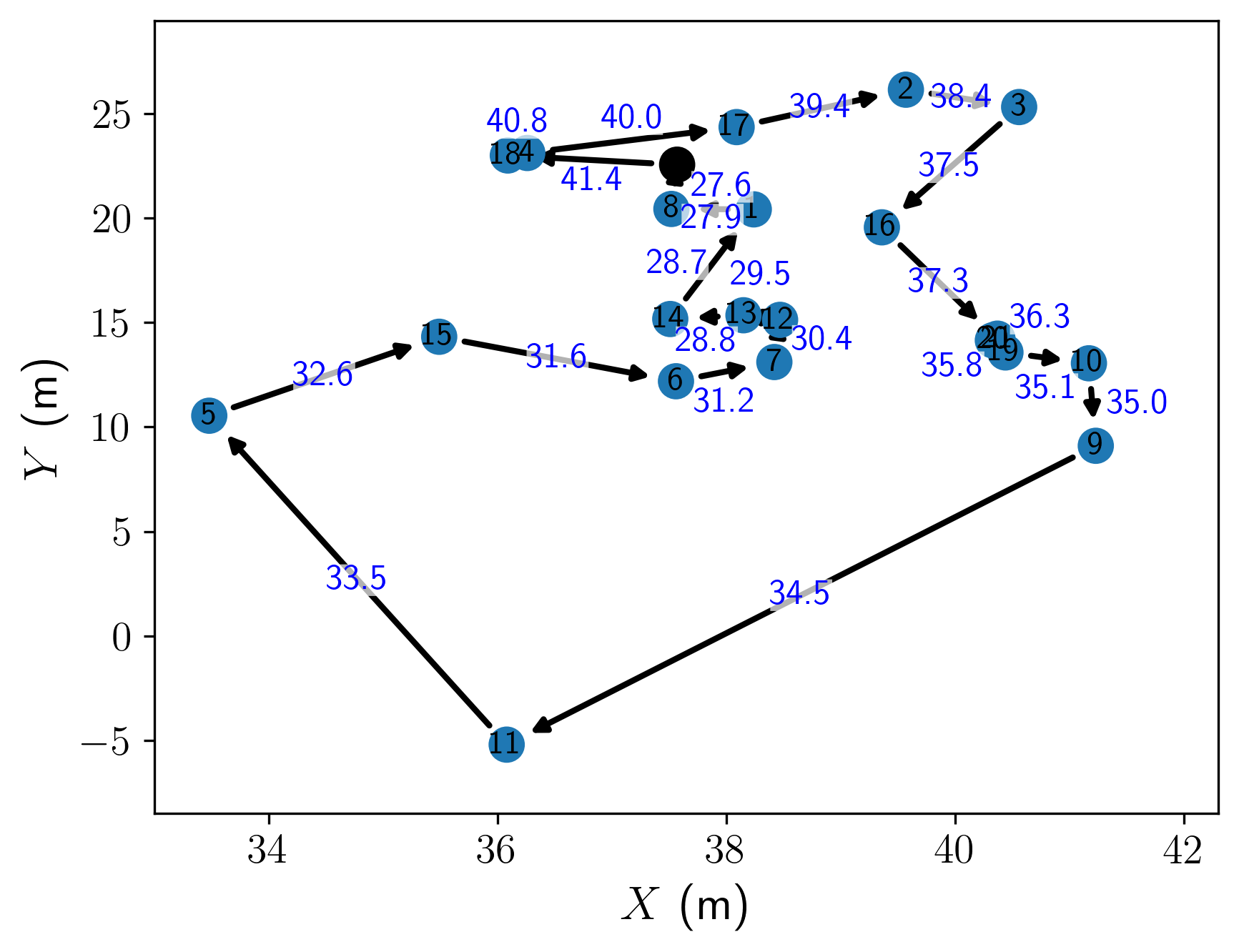}
    \label{subfig: optimal_tsp_tour}}
    \hfil
    \subfloat[Optimal tour obtained from proposed formulation (cost = $246.15$)]{\includegraphics[width = 0.45\textwidth]{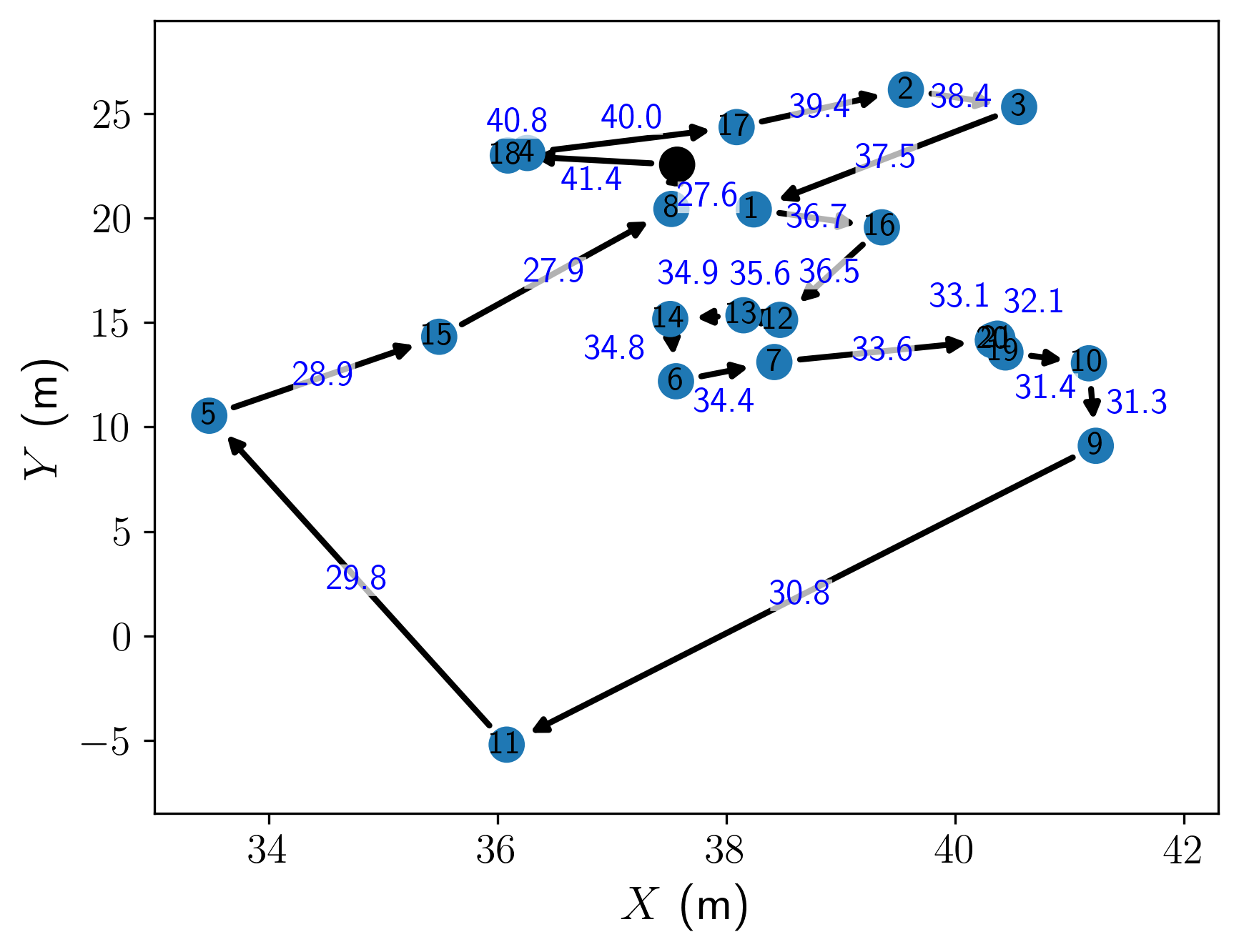}
    \label{subfig: optimal_tour}}
    \caption{Optimal solution vs optimal TSP solution for ulysses22 with $\gamma = 10$. In both figures, the depot location is shown with a shaded black node.}
    \label{fig: load_tsp_vs_tsp}
\end{figure}

Nevertheless, the solution derived from the TSP is fairly close to the optimal LD-TSP solution, as shown in Fig.~\ref{fig: difference_LKH_optimal}. The cost difference is at most $5\%$. Therefore, users needing a quick, feasible solution can generate two target sequences using LKH~\cite{LKH_TSP_solver} and select the one with the lower LD-TSP cost.

\begin{figure}[htb!]
	\centering
	\includegraphics[width=0.8\linewidth]{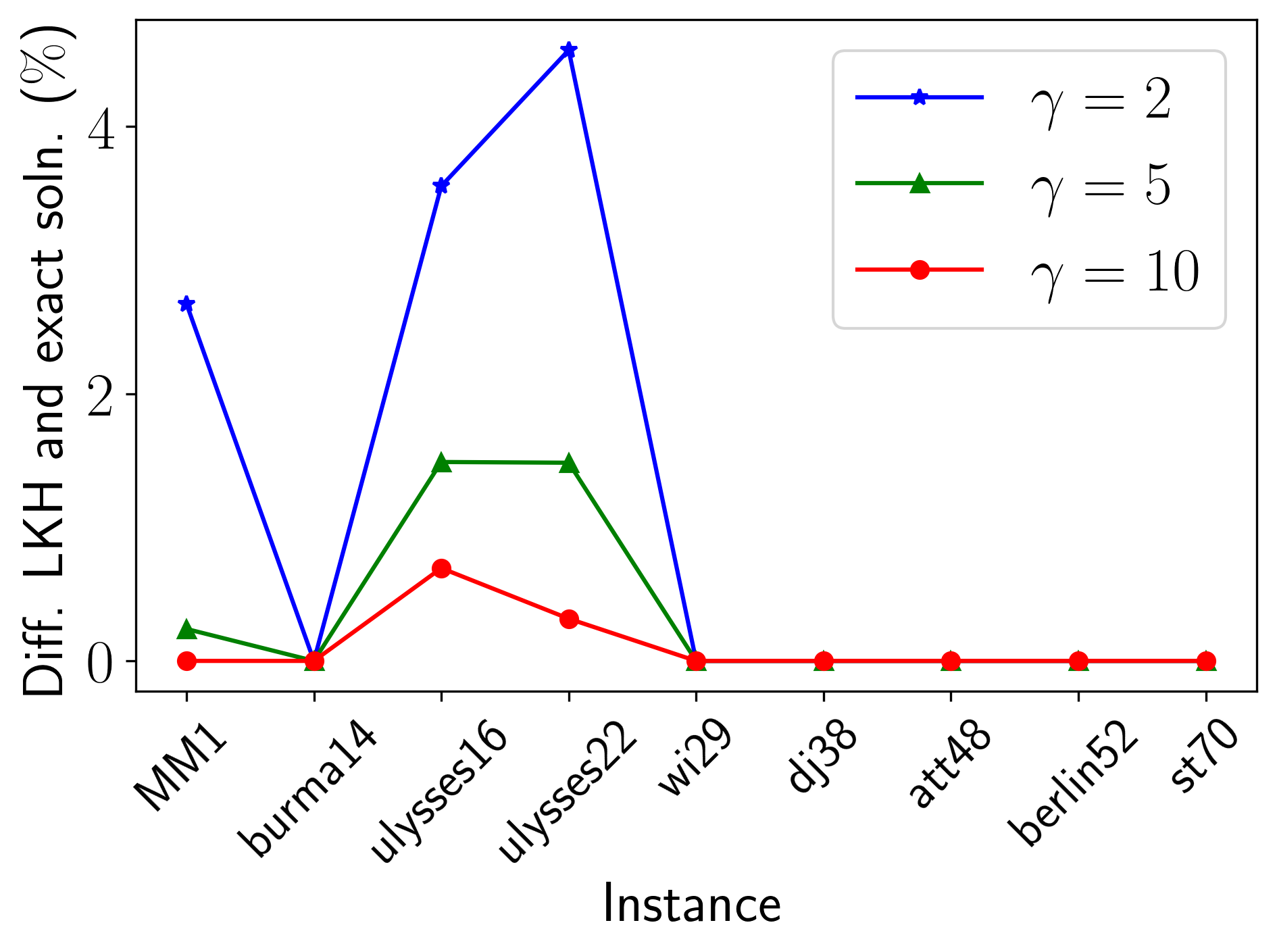}
    \caption{Percentage difference between LKH solution and optimal solution/best incumbent solution, defined by $\frac{\text{LKH} - \text{optimal/incumbent}}{\text{optimal/incumbent}} \times 100$}
    \label{fig: difference_LKH_optimal}
\end{figure}


\subsection{Benchmarking against state-of-the-art}

We benchmark our formulation against the hazmat TSP from~\cite{load_sequence_risk}. The authors considered three standard vehicle routing problem instances, sourced from~\cite{VRP_instances}, along with smaller subsets. For each edge, risk values were sampled from $[0, 100]$. As the authors' datasets are unavailable, we use the same three standard instances and generate three variations for each. To recover the hazmat TSP from our LD-TSP, we set $\gamma = 0$ in our MILP. We then compare the optimality gap of our MILP with the results reported in~\cite{load_sequence_risk}, using the same computation time limit of one hour. We remark here that the hardware specifications of the machine used in \cite{load_sequence_risk} is similar to our hardware.

Our results are shown in Table~\ref{tab: hazmat_TSP}. In two out of three instances, we achieve a substantially smaller optimality gap, reducing it from nearly $90\%$ to about $50\%$. For the third instance, our gap is comparable to that reported in~\cite{load_sequence_risk}. Thus, our LD-TSP formulation generalizes effectively to the hazmat TSP.

For the same instances, we also consider $\gamma = 2$ to evaluate the performance of our MILP on the LD-TSP. 
The optimal solution was found in at most about 6 minutes for all cases, as shown in Table~\ref{tab: hazmat_TSP}. This further demonstrates the tightness of our formulation for the LD-TSP.

\begin{table*}[htb!]
\centering
\caption{Benchmarking on hazmat TSP instances and the algorithm proposed in \cite{load_sequence_risk}}
\label{tab: hazmat_TSP}
\begin{tabular}{|c|c|c|c|}
\hline
Instance & \begin{tabular}[c]{@{}c@{}}Benchmark gap ($\%$)\end{tabular} & \begin{tabular}[c]{@{}c@{}}MILP gap for $\gamma = 0$ ($\%$)\end{tabular} & \begin{tabular}[c]{@{}c@{}}Compute time for $\gamma = 2$ (s)\end{tabular} \\ \hline
A-n32-k5 & 88.6 & [48.4, 53.2, 54.9] & [19.55, 41.12, 45.13] \\ \hline
B-n38-k6 & 89.8 & [51.5, 47.9, 54.1] & [313.67, 89.98, 379.01] \\ \hline
P-n21-k2 & 28.2 & [32.4, 29.2, 34.9] & [4.13, 2.79, 8.69] \\ \hline
\end{tabular}
\end{table*}

\section{Conclusion}

In this article, we present a novel mixed-integer linear programming (MILP) formulation for the load-dependent traveling salesman problem (LD-TSP), where energy consumption depends on the vehicle’s current load. The considered problem has extensive applications, such as package delivery using cargo drones. A key challenge is modeling the energy consumption for travel between locations. We show that minimizing energy is equivalent to minimizing distance, with energy varying linearly with vehicle mass. Based on this, we first develop a mixed-integer nonlinear program for LD-TSP and then relax it to obtain an MILP. Our approach solves instances with up to 50 targets in under one minute and achieves solutions within $13\%$ optimality for instances up to 70 targets in 10 minutes. 
Benchmarking against three baseline formulations and an algorithm from a related problem, our MILP consistently outperforms all baselines in extensive computational experiments.


\section{Acknowledgment}

The authors thank Dr. Srinivas Chippada for useful discussions regarding the equivalence of time-optimal and energy-optimal paths.

\bibliographystyle{IEEEtran}
\bibliography{cite}

\appendix
\subsection{Derivation of power consumption}

We make the following assumptions:  
(i) The vehicle travels at a constant speed $V_0$ and altitude (for UAVs) in a steady wind, which, without loss of generality, blows eastward at constant speed $v_w > 0$ (relative to the ground).  
(ii) The drag force ${\bf F}_d$ depends only on speed and opposes the vehicle's longitudinal motion.  
(iii) The vehicle's inertia is neglected.  
(iv) The wind does not overpower the vehicle, i.e., $|v_w| < V_0$; otherwise, controllability issues arise, and the vehicle may not reach westward locations.
(v) For UAVs, the lift-to-drag ratio is constant; this assumption is not relevant for ground vehicles. 
(vi) For ground vehicles, such as trucks, aerodynamic drag is negligible compared to rolling resistance. 
(vii) The thrust direction remains fixed, so frontal area and drag coefficients do not change.  
Fig.~\ref{fig: model_cargo_drone} illustrates the problem setup.

\begin{figure}[htb!]
    \centering
    \includegraphics[width = 0.5\linewidth]{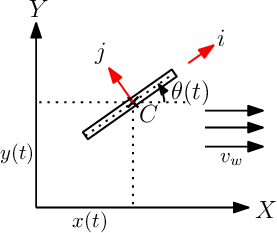}
    \caption{Model of cargo drone}
    \label{fig: model_cargo_drone}
\end{figure}

Let $I, J$ denote the unit vectors along $X$ (east) and $Y$ (north) directions of a ground Cartesian frame. Let $i, j$ be unit vectors along the longitudinal and lateral directions of the drone. Let the coordinate pair $(x, y)$ denote the location of the drone in the ground Cartesian frame and $\theta$ be its heading angle (relative to the wind). One may express the kinematic equations as given in \eqref{eq:kinematics}. 
We remark here that since we consider a vehicle that can turn in place, we do not consider a constraint on the control input $u$ for the turn rate. Nevertheless, the model holds even in the case of vehicles with a curvature constraint.

The constitutive equation for the drag force is:
\begin{align} \label{eq: drag_force}
{\bf F}_d = \frac{1}{2} \rho C_d A_f V_0^2 i = \phi(V_0) i,
\end{align}
where $\rho, C_d,$ and $A_f$ are the density of air, coefficient of drag of the vehicle, and frontal area, respectively.

Note that the velocity of the vehicle may be expressed as 
\begin{align}
{\bf v} (t) = \dot x (t) I + \dot y (t) J = V_0 i + v_w I. 
\end{align}

Using the derived dissipative force, the power expended can be derived as shown in \eqref{eq: power}.

\textbf{Remark:} While this derivation holds for a UAV, 
for a ground vehicle (e.g., a truck),~\eqref{eq: drag_force} is replaced by $\mathbf{F}_d = Mgf_r$, where $f_r$ is the constant rolling resistance coefficient and $M$ is the nominal mass. The rolling resistance force will dominate aerodynamic drag for such a vehicle. With this change, the equivalence between energy and time-optimality still holds, and the cost (energy) remains linearly dependent on vehicle mass.

\textbf{Remark:} Though the derivation assumes the dissipative force to be linearly dependent on the vehicle's mass, the equivalence of time-optimality and energy-optimality shown in Theorem~\ref{theorem} holds for an affine function as well.

\end{document}